\documentclass[10pt,letterpaper,oneside]{article}

\author{Nathan Grigg\thanks{Nathan Grigg and Nathan Manwaring are undergraduates at Brigham Young University. This research was funded in part by a grant from the BYU Office of Research and Creative Activities.} \\ Nathan Manwaring}
\title{An Elementary Proof of the Fundamental Theorem of Tropical Algebra}
\usepackage{amssymb,amsmath,amsthm}

\newtheorem*{fta}{Fundamental Theorem of Tropical Algebra}
\newtheorem{thm}{Theorem}[section]
\newtheorem{lem}[thm]{Lemma}
\theoremstyle{definition}
\newtheorem{defn}[thm]{Definition}
\newtheorem*{nte}{Note}

\newcommand{\Q}{\mathcal Q}
\newcommand{\+}{\oplus}
\newcommand{\half}{\frac{1}{2}}
\newcommand{\x}{}
\newcommand{\y}{}

\usepackage{url}
\usepackage[pdftex,bookmarks]{hyperref}

\begin{document}
\maketitle
\begin{abstract} In this paper we give an elementary proof of the Fundamental Theorem of Algebra for polynomials over the rational tropical semi-ring. We prove that, tropically, the rational numbers are algebraically closed. We provide a simple algorithm for factoring tropical polynomials of a single variable. A central idea is the concept of \emph{least-coefficient polynomials} as representatives for classes of functionally equivalent polynomials. This idea has importance far beyond the proof of the Fundamental Theorem of Tropical Algebra.
\end{abstract}

\section{Introduction}
In this paper we will consider the tropical semi-ring, as discussed by Richter-Gebert, Sturmfels, and Theobald in~\cite{richter:first} and by Speyer and Sturmfels in~\cite{speyer}. Our goal is to give an elementary proof of the Fundamental Theorem of Algebra as it applies to the tropical semi-ring.

Although the authors of some papers refer to this theorem, they do not do more than confirm it as true or dismiss it as trivial. Nevertheless, our proof of this theorem is key to understanding vital components of the tropical algebraic structure. We note that one version of the proof has been published by Izhakian in~\cite{izhakian}, but Izhakian gives his proof over an ``extended'' tropical semi-ring that is substantially different from the standard tropical semi-ring that most others study. Hence, there is merit in discussing this elementary proof and the underlying ideas it addresses.

\begin{fta}
Every tropical polynomial in one variable with rational coefficients can be factored uniquely as a product of linear tropical polynomials with rational coefficients, up to functional equivalence.
\end{fta}

It is important to note that this theorem only applies up to functional equivalence. To illustrate this, note that we would factor $x^2\+4x\+6$ as $(x\+3)^2$. As functions, these are the same---for any $x$ they are equal. Nevertheless, the second expression expands to the polynomial $x^2\+3x\+6$, which is not the same polynomial as the first. For this reason, together with the fact that geometric properties of a polynomial depend only on its function, we will regard two polynomials as equivalent if they define the same function. For more information, see Section~\ref{sec:equality}.

Since we will be dealing with equivalence classes of polynomials, it is useful to have a representative for each functional equivalence class. In Section~\ref{sec:lcp}, we discuss one possible, very useful representative, called a \emph{least-coefficient polynomial}. We prove that every tropical polynomial is functionally equivalent to a least-coefficient polynomial and that each least-coefficient polynomial can be easily factored using the formula given in Section~\ref{sec:fta}.

\begin{defn}
The \emph{rational tropical semi-ring} is $\Q = (\mathbb Q\cup \infty,\+,\odot)$\,, where
\begin{eqnarray*} \label{eqn:addition and multiplication}
 a \+ b &:=& \min(a,b)\,,\ \textup{and} \\
 a \odot b&:=&a+b\,.
\end{eqnarray*}
\end{defn}

\noindent We note that the additive identity of $\Q$ is $\infty$ and the multiplicative identity is $0$\,.
Elements of $\Q$ do not have additive inverses, but the multiplicative inverse of $a$ is the classical \emph{negative} $a$\,. The commutative, associative, and distributive properties hold.

\paragraph{Notation} We will write tropical multiplication $a \odot b$ as $ab$\,, and repeated multiplication $a \odot a$ as $a^2$\,. We will write classical addition, subtraction, multiplication, and division as $a+b$, $a-b$, $a\cdot b$, and $\frac{a}{b}$\,, respectively.

\section{Equality and Functional Equivalence}  \label{sec:equality}

A polynomial $f(x) \in \Q[x]$ is defined to be a formal sum
$$f(x)=a_nx^n \+ a_{n-1}x^{n-1} \+ \cdots \+ a_0\,.$$
For two polynomials $f$ and $g$\,, we write $f=g$ if each pair of corresponding coefficients of $f$ and $g$ are equal.

We can also think of a tropical polynomial as a function. Two polynomials are \emph{functionally equivalent} if for each $x\in\Q$\,, $f(x)=g(x)$\,. In this case, we write $f \sim g$\,. Notice that functional equivalence does not imply equality. For example, the polynomials $x^2\+ 1x \+2$ and $x^2\+2x\+2$ are functionally equivalent, but not equal as polynomials. In general, functional equivalence is a more useful equivalence relation to use with tropical polynomials than equality of coefficients.

\begin{defn} \label{defn:least coefficient}
A coefficient $a_i$ of a polynomial $f(x)$ is a \emph{least
coefficient} if for any $b \in \Q$ with $b<a_i$\,, the polynomial $g(x)$ formed by replacing $a_i$ with $b$ is not functionally equivalent to $f(x)$\,.
\end{defn}

\begin{nte}
If $f(x)=a_nx^n\+a_{n-1}x^{n-1}\+ \cdots \+ a_rx^r$\,, where $a_n,a_r\neq \infty$\,, then $a_n$ and $a_r$ are least coefficients. Additionally, if $r<i<n$ and $a_i=\infty$\,, then $a_i$ is not a least coefficient.
\end{nte}

\begin{lem}[Alternate definition of least coefficient]\label{lem:alt least coefficient}
Let $a_ix^i$ be a term of a polynomial $f(x)$\,, with $a_i$ not equal to infinity. Then $a_i$ is a least coefficient of $f(x)$ if and only if  there is some $x_0\in\mathbb Q$ such that $f(x_0)=a_ix_0^i$\,.\end{lem}

\begin{proof}
For all $x\in\mathbb Q$\,, note that $f(x)\leq a_ix^i$\,. Suppose that there is no $x$ such that $f(x)=a_ix^i$\,. Then $f(x)<a_ix^i$ for all $x$\,. Now, let $\varphi (x)=f(x)-(i\cdot x+a_i)$\,. Note that $\varphi$ is a piecewise-linear, continuous function that is linear over a finite number of intervals. Thus, there is an interval large enough to contain all the pieces of $\varphi$. By applying the extreme value theorem to this interval, we see that $\sup \varphi \in \varphi(\mathbb R)$\,, and hence $\sup \varphi <0$\,. Let $\epsilon=|\sup \varphi|$ and $b\in\Q$ be such that $a_i-\epsilon<b<a_i$\,. Then
$$f(x)-(i\cdot x+b)<f(x)-(i\cdot x+a_i)+\epsilon \leq0$$
and therefore $f(x)< i\cdot x+b$
for all $x\in \mathbb Q$\,. Therefore, the polynomial created by replacing $a_i$ with $b$ is functionally equivalent to $f(x)$\,, so $a_i$ is not a least coefficient.

For the other direction, suppose that  there is an $x_0\in\mathbb Q$
such that $f(x_0)=a_ix_0^i$\,. Given $b<a_i$\,, let $g(x)$ be $f(x)$ with $a_i$ replaced by $b$\,.
Then $g(x_0)\leq bx_0^i<a_ix_0^i=f(x_0)$\,, so $g$ is not functionally equivalent to $f$\,. Therefore
$a_i$ is a least coefficient.
\end{proof}

\section{Least-coefficient polynomials}
\label{sec:lcp}

\begin{defn} \label{defn:least coefficients polynomial}
A polynomial is a \emph{least-coefficient polynomial}
if all its coefficients are least coefficients.
\end{defn}

\begin{lem}[Uniqueness of least-coefficient polynomials] \label{lem:lcp equality}
Let $f$ and $g$ be least-coefficient polynomials. Then $f$ is equal to $g$ if and only if $f$ is functionally equivalent to $g$\,.
\end{lem}

\begin{proof}
It is clear that $f=g$ implies $f\sim g$\,. For the other direction, suppose that $f\neq g$\,. Then for some term $a_ix^i$ of $f(x)$ and the corresponding term $b_ix^i$ of $g(x)$\,, we have $a_i\neq b_i$\,. Without loss of generality, suppose $a_i<b_i$\,. Since $g$ is a least-coefficient polynomial, $g(x_0)=b_ix_0^i$ for some $x_0$\,, by Lemma \ref{lem:alt least coefficient}. Now,
$$f(x_0)\leq a_ix_0^i<b_ix_0^i=g(x_0)\,,$$
so $f$ is not functionally equivalent to $g$\,.
\end{proof}

We will now prove that every functional equivalence class contains a unique least-coefficient polynomial. This least-coefficient representative is often the most useful way to represent a functional equivalence class of tropical polynomials.

\begin{lem} \label{lem:unique LCP}
Let $f(x)=a_nx^n \+ a_{n-1}x^{n-1} \+ \cdots \+ a_rx^r$\,. There is a unique least-coefficient polynomial $g(x)=b_nx^n\+b_{n-1}x^{n-1}\+\cdots\+ b_rx^r$ such that $f\sim g$\,. Furthermore, each coefficient $b_j$ of $g(x)$ is given by
\begin{equation} \label{eqn:LCP algorithm}
b_j=\min\left(\bigg\{a_j\bigg\}\cup \left\{\frac{a_i\cdot (k-j)+a_k\cdot
(j-i)}{k-i}\bigg| r\leq i<j<k \leq n\right\}\right)\,.
\end{equation}
\end{lem}

\begin{proof}
First we will show that $f\sim g$\,.
Given $x_0$\,, note that $f(x_0)=a_sx_0^s=a_s+s\cdot x_0$ for some $s$\,. Also,
\begin{eqnarray}
g(x_0)&=&\min_{r\leq j\leq n} \left\{b_j+j\cdot x_0 \right\}\nonumber \\
 \label{eqn:algorithm value}
&=&\min_{r\leq i<j<k \leq n}
  \left\{a_j + j \cdot x_0,\frac{a_i\cdot (k-j)+a_k\cdot (j-i)}{k-i}+j\cdot x_0\right\}
\,.
\end{eqnarray}
\renewcommand{\x}{ \frac{a_i-a_k}{k-i}}
\renewcommand{\y}{\frac{a_i\cdot(k-j)+a_k\cdot(j-i)}{k-i}}
So for any $i, j, \text{and } k$ such that $r\leq i<j<k\leq n$\,, if $x_0\geq \x$ then
\begin{eqnarray*}
a_s+s\cdot x_0 &\leq& a_i+i \cdot x_0\\
&=&             \y+(j-i)\cdot \left(\x\right)+i \cdot x_0 \\
&\leq&         \y+(j-i)\cdot x_0 + i\cdot x_0 \\
&=&             \y+j\cdot x_0\,.
\end{eqnarray*}
A similar argument shows that  if $x_0 \leq \x$\,, then
\begin{equation*}
a_s+s\cdot x_0 \leq a_k+k\cdot x_0
\leq \y+j \cdot x_0\,.
\end{equation*}
Since this is true for all $i$, $j$, and $k$, the equation in (\ref{eqn:algorithm value}) evaluates to $g(x_0)=a_sx_0^s$, so $g(x_0)=f(x_0)$ and $f\sim g$\,, as desired.

Secondly, we must show $g$ is a least-coefficient polynomial. Given a coefficient $b_j$ in $g$\,, suppose that $a_j$ is a least coefficient of $f$\,.  From
Equation (\ref{eqn:LCP algorithm})
we see that $b_j\leq a_j$\,. Since $a_j$ is a least coefficient, there is some $x_0$ such that $f(x_0)=a_jx_0^j$\,, so $b_jx_0^j \geq g(x_0)=f(x_0)=a_jx_0^j$\,. Therefore $b_j=a_j$ and $g(x_0)=b_jx_0^j$\,.

Now suppose that $a_j$ is not a least coefficient.  Then since $a_r$ and $a_n$ are least-coefficient, we can choose $u<j$ and $v>j$ such that $a_u$ and $a_v$ are least coefficients and for any $t$ such that $u<t<v$\,, $a_t$ is not a least coefficient.

\renewcommand{\x}{\frac{a_u-a_v}{v-u}}
\renewcommand{\y}{\frac{a_u\cdot(v-j)+a_v \cdot (j-u)}{v-u}}

Let $x_0=\x$ and suppose, by way of contradiction, that $f(x_0)\neq a_ux^u_0$\,.  Then $f(x_0)=a_wx_0^w<a_ux_0^u$ for some $w$\,. Note that $a_w$ is a least coefficient, so it cannot be that $u<w<v$ by our assumption on $u$ and $v$\,. If $w<u$ then for $x\geq x_0$\,,
\begin{eqnarray*}
a_w+w\cdot x &=&a_w+u\cdot x - (u-w) \cdot x \\
&\leq& a_w+u\cdot x - (u-w) \cdot x_0 \\
&=&a_w+w\cdot x_0 + u \cdot (x-x_0) \\
&<&a_u + u \cdot x_0 + u \cdot (x-x_0)\\
&=&a_u+u\cdot x
\end{eqnarray*}
For $x<x_0$\,,
\begin{eqnarray*}
a_v+v\cdot x &=&a_v+u\cdot x + (v-u)\cdot x\\
&<&a_v+u \cdot x +(v-u) \cdot x_0 \\
&=&a_v+v\cdot x_0+u\cdot (x-x_0)\\
&=&a_u+u\cdot x_0 + u\cdot (x-x_0)\\
&=&a_u+u\cdot x
\end{eqnarray*}
So there is no $x$ such that $f(x)=a_ux^u$ and thus $a_u$ is not a least coefficient, which contradicts our assumption. If $w>v$\,, a similar argument shows that $a_v$ is not a least coefficient, again contradicting our assumption. Therefore,
\begin{eqnarray}
f(x_0)&=&a_u+u\cdot\left(\x\right) \nonumber\\
&=&\y + j \cdot \left(\x\right) \nonumber\\
\label{eqn:c}&=&c+j\cdot x_0\,\text{, where } c=\y
\end{eqnarray}
Again, from (\ref{eqn:LCP algorithm}) we see that $b_j \leq c$\,, and from (\ref{eqn:c}) we see $cx_0^j=f(x_0)=g(x_0)\leq b_jx_0^j$\,. So $b_j=c$ and $g(x_0)=b_jx_0^j$\,.

Finally, $g$ is the only such polynomial by Lemma \ref{lem:lcp equality}.
\end{proof}

\begin{nte}
The use of a least-coefficient polynomial as a best representative for a functional equivalence class is one of the key ideas of this paper. We cannot develop well-defined algebraic transformations of tropical polynomials without unique representatives for functional equivalence classes. While Izhakian discusses in~\cite{izhakian} what he calls an ``effective'' coefficient (similar to a least coefficient), the idea of using least-coefficient polynomials to represent functional equivalence classes has not been discussed.
\end{nte}

\begin{lem} \label{lem:LCP Criteria}
Let $f(x)=a_nx^n \+ a_{n-1}x^{n-1} \+ \cdots \+ a_rx^r$\,, where each $a_i$ is not infinity. Let $d_i=a_{i-1}-a_i$ be the difference between two consecutive coefficients. Then $f(x)$ is a least-coefficient polynomial if and only if the
difference between consecutive coefficients is non-decreasing, that is, if
$d_n\leq d_{n-1} \leq \cdots \leq d_{r+1}\,.$
\end{lem}

\begin{proof}
Suppose that $f$ has a set of consecutive coefficients whose differences are \emph{decreasing}, that is, $ax^{i+1}$, $bx^{i}$, and $cx^{i-1}$ are consecutive terms of $f(x)$ such that $b-a>c-b$\,. Then $b>\half \cdot (a+c)$\,. We will show that $f(x_0)<bx_0^i$ for all $x_0$\,, meaning that $b$ is not a least coefficient.

Given $x_0$\,, if $x_0\leq\half \cdot (c-a)$ then
\begin{eqnarray*}
ax_0^{i+1}&=&(i+1)\cdot x_0 + a\\
 &\leq& i \cdot x_0 + \half \cdot (c-a) + a\\
&=&i \cdot x_0 + \half \cdot (c+a)\\
&<& i \cdot x_0+b=bx_0^i\,,
\end{eqnarray*}
so $f(x_0)\leq ax_0^{i+1}<bx_0^i$\,.
Similarly, if $x_0\geq\half \cdot (c-a)$\,,
\begin{eqnarray*}
cx_0^{i-1}&=&(i-1)\cdot x_0 + c\\
& \leq& i \cdot x_0 - \half\cdot (c-a) +c\\
&=& i \cdot x_0 + \half \cdot (c+a)\\
&<&i\cdot x_0+b=bx_0^i\,,
\end{eqnarray*}
so $f(x_0)\leq cx_0^{i-1}<bx_0^i$\,. Therefore $b$ is not a least coefficient, and $f$ is not a least-coefficient polynomial.

For the other direction, suppose that the differences between the coefficients of $f(x)$ are
nondecreasing. Since $a_n,a_r\neq\infty$\,, $a_n$ and $a_r$ are least coefficients. Let $a_i$ be a coefficient of $f$\,, with $r<i<n$\,, and let $x_0=\frac{a_{i-1}-a_{i+1}}{2}$\,. We will show that $f(x_0)=a_ix_0^i$\,, so $a_i$ is a least coefficient. We must show for all $k$ that
$i\cdot x_0+a_i\leq k\cdot x_0+a_k$\,.

This is certainly true for $i=k$\,. Suppose $k>i$\,. Then, since $(a_t-a_{t+1})\leq (a_s-a_{s+1})$ for $t\geq s$\,, we have
\begin{eqnarray*}
(a_i-a_{i+2})=(a_i-a_{i+1})+(a_{i+1}-a_{i+2})&\leq&2\cdot (a_i-a_{i+1})\\
(a_i-a_{i+3})=(a_i-a_{i+2})+(a_{i+2}-a_{i+3})&\leq&3\cdot
(a_i-a_{i+1})
\end{eqnarray*}
And in general we get
\begin{eqnarray*}
(a_i-a_k)&\leq&(a_{i}-a_{i+1})\cdot (k-i)=\frac{1}{2}\cdot \big(2\cdot (a_i-a_{i+1})\big)\cdot (k-i)\\
&\leq&\frac{1}{2}\cdot \big(
(a_{i-1}-a_i)+(a_i-a_{i+1})\big)\cdot (k-i)=x_0\cdot (k-i)\,.
\end{eqnarray*}

Thus, $i\cdot x_i+a_i\leq k\cdot x_i+a_k$\,. A similar argument
holds for $k<i$\,. So (tropically) $a_ix^i_i \leq a_sx^s_i$ for all
$s$\,. This means that $f(x_i)=a_ix^i_i$, so $a_i$ is a least coefficient. Therefore, $f$ is a least-coefficients polynomial.
\end{proof}

\begin{nte}
If $f(x)$ has a coefficient $a_i$ such that $a_i=\infty$ for $r<i<n$\,, then $f$ is not a least-coefficient polynomial; but of course, $a_i=\infty$ for all $i>n$ and all $i<r$\,, even in a least-coefficient polynomial.
\end{nte}

\section{The Fundamental Theorem of Tropical Algebra}
\label{sec:fta}

\begin{fta} 
Let $f(x)=a_nx^n\+
a_{n-1}x^{n-1}\+ \cdots \+ a_rx^r$ be a least
coefficients polynomial. Then $f(x)$ can be written uniquely as the product
of linear factors
\begin{equation}\label{eqn:factored polynomial}
a_nx^r\left(x\+ d_n\right)
\left(x\+ d_{n-1}\right)\cdots \left(x\+ d_{r+1}\right)\,,
\end{equation}
where $d_i=a_{i-1}- a_i$\,. In other words, the roots of $f(x)$
are the differences between consecutive
coefficients.
\end{fta}

\begin{proof}
Since $f(x)$ is a least-coefficient polynomial, the differences between consecutive coefficients is non-decreasing, i.e.,
$d_n\leq d_{n-1}\leq \cdots \leq d_{r+1}$\,.
Knowing these inequalities, we can expand (\ref{eqn:factored polynomial}) to get
\begin{equation}\label{eqn:expanded polynomial}
a_nx^n\+ a_nd_nx^{n-1}\+ a_nd_nd_{n-1}x^{n-2}\+
\cdots \+ a_nd_nd_{n-1}\cdots d_{r+1}\,.
\end{equation}
But the coefficient of the $x^{i}$ term in this polynomial is
\begin{equation*}
a_nd_nd_{n-1}\cdots d_{i+1}=a_n+d_n+d_{n-1}+ \cdots + d_{i+1}\,.
\end{equation*}
A straightforward computation shows that this is equal to $a_i$, so the polynomial in (\ref{eqn:expanded polynomial}) is equal to $f(x)$, as desired.

Now suppose that there is another way of writing $f(x)$ as a product of linear factors. Call this product $g$ and note that it must have the same degree as $f$\,. Additionally, the smallest non-infinite term of $g$ must have the same degree as the smallest non-infinite term of $f$\,. Hence, we are able to write
$g(x)=a'_nx^r\left(x\+ d'_n\right)
\left(x\+ d'_{n-1}\right)\cdots \left(x\+ d'_{r+1}\right)$\,,
with each $d'_i$ chosen, after reindexing, if necessary, such that $d'_n\leq d'_{n-1}\leq \cdots\leq d'_{r+1}$\,. Expanding this product shows that the differences between consecutive coefficients of $g$ are non-decreasing, so $g$ is a least-coefficient polynomial by Lemma \ref{lem:LCP Criteria}. We see from (\ref{eqn:expanded polynomial}) that $f\neq g$\,, so by Lemma \ref{lem:lcp equality}, $f$ is not functionally equivalent to $g$\,. Therefore, the factorization is unique.
\end{proof}

Finally we note that tropical factoring gives us a slightly different result than classical factoring. Classically, the set of roots (or zero locus) of a polynomial is the set of points at which the polynomial evaluates to the additive identity.  Unfortunately, tropical polynomials have either no roots or trivial roots in this sense.  In fact, if $f(x)\neq \infty$\,, then $f(x_0)$ never evaluates to the additive identity $\infty$ when $x_0\neq \infty$\,.  However, as we have seen, polynomials in $\Q[x]$ can be factored and seem to have ``roots,'' although they do not evaluate to the additive identity at these points.  Clearly, we must use a different, more meaningful definition. In~\cite{richter:first} motivation is given for the following definition of zero locus.

\begin{defn}
Let $f(x)\in \mathcal{Q}[x]$.  The \emph{tropical zero locus} (or \emph{corner locus}) $\mathcal{Z}(f)$ is the set of points $x_0$ in $\mathcal{Q}$ for which at least two monomials of $f$ attain the minimum value.
\end{defn}

The $d_i$ in (\ref{eqn:factored polynomial}) are precisely that points of $\mathcal Z (f)$, as we now show.

\begin{thm} \label{thm:zero locus}
Given a point $d\in \Q$ and a least-coefficient polynomial $f(x)$\,, $x\+d$ is a factor of $f(x)$ if and only if $f(d)$ attains its minimum on at least two monomials.
\end{thm}

\begin{proof}
First, suppose that $x\+d$ is a factor of $f(x)$\,. If we write $f$ as a product of linear factors as in (\ref{eqn:factored polynomial}), $d_i=d$ for some $d_i$\,.For all $j<i$\,,
\begin{eqnarray*}
a_i+i\cdot d_i&=&a_i+(i-j)\cdot d_i+ j\cdot d_i\\
&\leq&a_i+\underbrace{d_i+d_{i-1} + \cdots + d_{j+1}}_{i-j\text{ terms}} +j\cdot d_i\\
&=&a_j+j\cdot d_i\,.
\end{eqnarray*}
A similar calculation shows that for $j>i$\,, we have $a_i+i\cdot d_i\leq a_j+j\cdot d_i$\,. So $f(d_i)=a_id_i^i$\,.
Also, $a_id_i^i=a_nd_nd_{n-1}\cdots d_{i+1}d_i^{i} = a_nd_nd_{n-1}\cdots d_{i} d_i^{i-1}=a_{i-1}d_i^{i-1}$\,,
so the minimum is attained by at least two monomials of $f(x)$ at $x=d$\,.

For the other direction, suppose that the minimum is attained by two monomials at $f(d)$\,. By way of contradiction, suppose that these monomials are not consecutive. Then for some $j<i<k$\,, we have $a_jd^j=a_kd^k<a_id^i$\,. If $x\leq d$ then
\begin{eqnarray*}
a_k+k\cdot x &=&a_k+i\cdot x+(k-i)\cdot x \leq a_k+i\cdot x + (k-i)\cdot d\\
&=&a_k+k\cdot d+i\cdot (x-d)<a_i+i\cdot d + i\cdot (x-d)\\
&=&a_i+i\cdot x
\end{eqnarray*}
Similarly, if $x\geq d$\,, then $a_jx^j<a_ix^i$\,. Thus there is no $x$ such that $f(x)=a_ix^i$\,, so $a_i$ is not a least coefficient, which is a contradiction.  Therefore there is some $i$ such that $a_id^i=a_{i-1}d^{i-1}$\,. Thus we have
\begin{equation*}
0=a_{i-1}+(i-1)\cdot d-(a_i+i\cdot d)=a_{i-1}-a_i-d\,.
\end{equation*}
So $d=a_{i-1}-a_i$\,, the difference between two consecutive coefficients. Since $f$ is a least-coefficient polynomial, $x\+d$ is a factor of $f$ by the Fundamental Theorem\,.
\end{proof}

Thus, as in the classical case, the unique factorization of a polynomial in $\mathcal{Q}[x]$ gives us what could be considered the \emph{roots} of the polynomial.  It is clear that all of the arguments and results of this paper hold if we replace the rationals $\mathbb Q$ with any ordered field. Thus any ordered field, together with $\infty$\,, can be said to be tropically algebraically closed.

\nocite{thesis}
\phantomsection
\addcontentsline{toc}{section}{References}
\bibliographystyle{amsplain} 
\bibliography{fta}

\end{document}